\documentclass[12pt]{amsart}
\usepackage{amsfonts}

\newtheorem{theorem}{Theorem}
\newtheorem{question}[theorem]{Question}

\newtheorem{proposition}[theorem]{Proposition}
\newtheorem{example}[theorem]{Example}

\newcommand{\bC}{\mathbb{C}}

\newcommand{\bN}{\mathbb{N}}

\newcommand{\cV}{\mathcal{V}}

\newcommand{\tr}{{\rm tr}\,}
\newcommand{\ran}{{\rm ran}\,}

\newcommand{\supp}{{\rm \, supp \,}}

\begin{document}

\baselineskip 6.5mm

\title{Triangularizability of trace-class operators with increasing spectrum} 

\author{Roman Drnov\v sek}

\date{\today}

\begin{abstract}
\baselineskip 6.5mm

For any measurable set $E$ of a measure space $(X, \mu)$, let $P_E$ be the (orthogonal) projection on the Hilbert space 
$L^2(X, \mu)$ with the range $\ran P_E = \{ f \in L^2(X, \mu) : f = 0 \ \textrm{a.e. on} \ E^c\}$  that 
is called a standard subspace of $L^2(X, \mu)$. 
Let $T$ be an operator on $L^2 (X, \mu)$ having increasing spectrum relative to standard compressions, that is, 
for any measurable sets $E$ and $F$ with $E \subseteq F$, the spectrum of the operator 
$P_E T|_{\ran P_E}$ is contained in the spectrum of the operator $P_F T|_{\ran P_F}$. 
In 2009, Marcoux, Mastnak and Radjavi asked whether the operator $T$  has a non-trivial invariant standard subspace. 
They answered this question affirmatively when either the measure space $(X, \mu)$ is discrete or the operator $T$ has finite rank. We study this problem in the case of trace-class kernel operators. 
We also slightly strengthen the above-mentioned result for finite-rank operators.
\end {abstract}

\maketitle
{\it Math. Subj.  Classification (2010)}:  47A15, 47A46,  47B10, 47B34, 47B65  \\

\baselineskip 6.6mm

\section{Introduction}

Let $\mu$ be a positive measure on a set $X$  such that $L^2(X, \mu)$  is a separable complex Hilbert space.
For each $\phi \in L^\infty(X, \mu)$, we define the multiplication operator $M_\phi$ on $L^2(X, \mu)$ by 
$M_\phi (f) = \phi f$ \ ($f  \in L^2(X, \mu)$). 
Let  $M^\infty(X, \mu) = \{M_\phi : \phi \in L^\infty(X, \mu)\}$ be the collection of all multiplication operators. 
This collection is a maximal abelian, selfadjoint algebra (usually abbreviated as {\it masa}) in the Banach algebra of all (bounded) operators on $L^2(X, \mu)$. 
The spectrum and the spectral radius of an operator $T$  on $L^2(X, \mu)$ are denoted by $\sigma(T)$ and $r(T)$, respectively.  
An operator $T$  on $L^2(X, \mu)$ is said to be {\it positive} if it maps nonnegative functions to nonnegative ones.

An operator $P$ on the Hilbert space $L^2(X, \mu)$ is called a {\it standard projection corresponding to a measurable set} 
$E \subseteq X$ if it is the multiplication operator by the characteristic function $\chi_E$ of $E$. 
In this case its range $\ran P$ can be identified with the Hilbert space $L^2(E, \mu|_E)$, and it is said to be 
a {\it standard subspace}; in the Banach lattice theory it is called a closed ideal (or a band) of $L^2(X, \mu)$. 
If such a subspace is non-trivial and invariant under an operator $T$ on $L^2 (X, \mu)$, we say that $T$ is {\it decomposable}
 (or, in the Banach lattice theory, $T$ is ideal-reducible).
 
We shall say that an operator $T$ on $L^2(X, \mu)$ {\it admits a standard triangularization} 
(or $T$ is completely decomposable or ideal-triangularizable)  
if we can find a totally ordered set $\Lambda$ and an increasing family $\{P_\lambda\}_{\lambda \in \Lambda}$ 
of standard projections such that $\{\ran P_\lambda\}_{\lambda \in \Lambda}$ is a maximal increasing family of standard subspaces that are all invariant under $T$. 
If the closed span of $\{P_\lambda\}_{\lambda \in \Lambda}$  in the weak operator topology is equal to $M^\infty(X, \mu)$, 
we say that $T$ admits a {\it multiplicity-free} standard triangularization.

Following \cite{MMR}, we say that an operator $T$ on $L^2 (X, \mu)$ {\it has increasing spectrum relative to standard compressions} if 
$$  \sigma (PT|_{\ran P}) \subseteq \sigma (QT|_{\ran Q}) $$
whenever $P$ and $Q$ are standard projections with \, $\ran P \subseteq \ran Q$. 
When this condition is required only for finite-dimensional standard projections $P$ and $Q$, 
the operator $T$ is said to {\it have increasing spectrum relative to finite-dimensional standard compressions}. 
The following two theorems are proved in \cite[Theorems 2.4 and 3.17]{MMR}.

\begin{theorem}
\label{discrete}
Let $\mu$ be the counting measure on a set $X$. If an operator $T$ on $L^2 (X, \mu)$ 
has increasing spectrum relative to finite-dimensional standard compressions, then 
it admits a standard, multiplicity-free triangularization.
\end{theorem}

\begin{theorem}
\label{finite_rank_MMR}
Let $T$ be an operator on $L^2(X, \mu)$ of rank $k  \in \bN$.
If $T$  has increasing spectrum relative to standard compressions, then it admits a standard, multiplicity-free triangularization. 
Furthermore, there is a chain of projections 
$$ 0 = P_0 < P_1 < \cdots < P_{3 k-1} < P_{3 k} = I $$
in $M^\infty(X, \mu)$, whose ranges are all invariant under $T$,  such that
$$ (P_j - P_{j-1}) T (P_j - P_{j-1}) = 0 $$
whenever $P_j - P_{j-1}$ has rank more than one.
\end{theorem}

Furthermore, \cite[Corollary 3.5]{MMR} provides an example of  a compact quasinilpotent operator 
that has increasing spectrum relative to standard compressions, but it only admits a standard triangularization that is not multiplicity-free.  

The authors of  \cite{MMR} left open the following question:

\begin{question}
\label{open} 
Suppose that $K$ is a compact operator on $L^2(X, \mu)$ that has increasing spectrum 
relative to standard compressions. Does $K$ admit a standard triangularization (which need not be multiplicity-free)? 
In particular, is $K$ decomposable?
\end{question}

It is worth mentioning that an affirmative answer to Question \ref{open} would extend the well-known de Pagter's theorem 
(see \cite{Pa86} or \cite[Theorem 9.19]{AbAl}) asserting that a positive compact quasinilpotent operator $K$ on $L^2(X, \mu)$ is decomposable. 
Namely, it is easy to show (see, e.g., the beginning of \cite[page 3520]{MMR})
that positivity of $K$ implies that the operator $P K P$ is quasinilpotent for each standard projection $P$, so that $K$ has increasing spectrum in this case. 

In this paper we consider Question \ref{open} in the case of trace-class kernel operators.
In Section 2 the underlying measure space is the unit interval $[0, 1]$, while in Section 3 we add atoms to it, and we prove the main result of the paper (Theorem \ref{main}). In Section 4 we slightly improve Theorem \ref{finite_rank_MMR} that is 
the main result of \cite{MMR}. In the rest of this section we recall some relevant definitions and facts.

An operator $K$ on $L^2(X, \mu)$ is called a {\it kernel operator} if there exists a measurable function 
$k  : X \times X \to \bC$ such that, for every $f \in L^2(X, \mu)$, the equality
$$ (Kf)(x) = \int_X k(x,y) f(y) d\mu(y)  $$
holds for almost every $x \in X$. The function $k$ is called the {\it kernel} of the operator $K$.
The kernel operator $K$ is positive if and only if its kernel $k$ is nonnegative almost everywhere.
If the kernel operator $K$ with kernel $k$ has the modulus $|K|$, then the kernel of $|K|$ is equal to $|k|$ almost everywhere.
For more details on kernel operators we refer the reader to  \cite{Za83}.

Given a compact operator $T$ on $L^2(X, \mu)$,  let $\{s_j(T)\}_j$ be a decreasing sequence of 
singular values of $T$, i.e.,  the square roots of the eigenvalues of the self-adjoint operator $T^* T$, 
where $T^*$ denotes the adjoint of $T$. If $\sum_j s_j(T) < \infty$, the operator $T$ is said to be a {\it trace-class} operator.
In this case, the {\it trace of} $T$ is defined by 
$$ \tr (T) = \sum_{n=1}^\infty \langle T f_n, f_n \rangle , $$
where $\{f_n\}_{n=1}^\infty$ is any orthonormal basis of $L^2(X, \mu)$.  By Lidskii's Theorem, the trace of a trace-class operator $T$ is equal to the sum of all eigenvalues of $T$ counting algebraic multiplicity.
We will use freely the equality $\tr (A T) = \tr (T A) $ that holds for every trace-class operator $T$ and every operator $A$.
For more details on trace-class operators see, e.g., \cite{RaRo}.

\section{The unit interval $[0,1]$}

Let $K$ be a kernel operator on $L^2[0,1]$ with a continuous kernel $k  : [0,1] \times [0,1] \to \bC$.
If $K$ is a trace-class operator, then the trace of $K$ is equal to the integral of its kernel along the diagonal:
$$ \tr (K) = \int_0^1 k(x,x) dx . $$
This equality is proved, for example, in \cite[Theorem 12 in Chapter 30]{La}.

The proof of the following theorem is a slight modification of the proof of   \cite[Theorem 3.6]{MMR}.
For $t \in [0,1]$, let $P_t$ denote the standard projection corresponding to the interval $[0, t]$.

\begin{theorem}
\label{quasi}
Let $K$ be a compact operator on $L^2[0,1]$. 
Suppose that $\sigma(P_t K P_t) \subseteq \sigma(K)$ for each $t \in [0,1]$ 
(in particular, if $K$ has increasing spectrum relative to standard compressions).
Then $K$ is quasinilpotent.
\end{theorem}

\begin{proof}
Assume that $r(K) > 0$. Since $K$ is compact, the map $t \mapsto P_t K P_t$ is continuous, and so the map  $\varphi: t \mapsto r(P_t K P_t)$
is continuous on $[0,1]$ (see, for example, \cite{Ne}). Since $\varphi(0)=0$ and $\varphi(1)=r(K)$, we conclude that
$[0, r(K)] \subseteq {\rm ran \, } \varphi$. 
Now, for each $t \in [0,1]$, there exists a complex number $\lambda_t \in \sigma(P_t K P_t)$ such that  
$| \lambda_t | = r(P_t K P_t) = \varphi(t)$.
Since $\sigma(P_t K P_t) \subseteq \sigma(K)$, this implies that $\{ \lambda_t :  t \in [0,1] \} \subseteq \sigma(K)$ 
contradicting the fact that $\sigma(K)$ is at most countable set.
Hence $K$ must be quasinilpotent.
\end{proof}

Following the definition in \cite{MMR}, we say that a kernel operator $K$ on $L^2[0,1]$ 
with a continuous kernel $k$ admits a {\it non-degenerate cycle of length} $n \ge 2$ 
if there exist $n$ pairwise distinct numbers $x_1$, $x_2$, $\ldots$, $x_n$ in $[0,1]$ such that
$$ k(x_1, x_2) k(x_2, x_3) k(x_3, x_4) \cdots k(x_{n-1}, x_n) k(x_n, x_1)  \neq 0 . $$
 
The following theorem provides a partial answer to Question \ref{open}. 

\begin{theorem}
\label{unit_interval}
Let $K$ be a trace-class kernel operator on $L^2[0,1]$ with a continuous kernel $k$. 
Suppose that $K$ has increasing spectrum relative to standard compressions and that the modulus $|K|$ is also a trace-class 
operator. Then $K$ and $|K|$ are quasinilpotent operators admitting a (common) standard triangularization.
\end{theorem}

\begin{proof}
By Theorem \ref{quasi}, $K$ is quasinilpotent.  
For $x \in [0,1) $ and $r \in (0,1-x)$, let $P_{x, r}$ be the standard projection corresponding to the interval $[x, x+r]$.
Since the operator $P_{x, r} K P_{x, r}$ is a quasinilpotent trace-class kernel operator on $L^2[0,1]$,
we have 
$$ 0 = \tr (P_{x, r} K P_{x, r}) = \int_x^{x+r} k(t,t) dt , $$
and so 
$$ 0 = \lim_{r \rightarrow 0} \frac{1}{r} \int_x^{x+r} k(t,t) \, dt = k(x, x) . $$

Now choose distinct $x, y  \in [0,1)$ and $r>0$ such that $[x, x+r]$ and $[y, y+r]$ are disjoint intervals 
contained in $[0,1]$. 
Since the operators $(P_{x, r}+P_{y, r}) K (P_{x, r}+P_{y, r})$, $P_{x, r} K P_{x, r}$ and 
$P_{y, r} K P_{y, r}$ are quasinilpotent trace-class kernel operators on $L^2[0,1]$,
we have 
$$ 0 = \tr (((P_{x, r}+P_{y, r}) K (P_{x, r}+P_{y, r}))^2)  = 
\tr ((P_{x, r}+P_{y, r}) K (P_{x, r}+P_{y, r}) K) = $$
$$ = \tr ((P_{x, r} K P_{x, r})^2) + \tr (P_{x, r} K P_{y, r} K) + 
\tr (P_{y, r} K P_{x, r} K ) + \tr ((P_{y, r} K P_{y, r})^2) = $$
$$ =  2 \cdot \tr (P_{x, r} K P_{y, r} K )  = 
2 \cdot \int_x^{x+r} ds \int_y^{y+r} k(s,t) k(t,s) dt . $$
Dividing by $r^2$ and letting $r$ tend to 0 we conclude that $k(x,y) k(y,x) = 0$ for all $x, y  \in [0,1)$.
Since $k$ is continuous, this holds for all $x, y  \in [0,1]$, and therefore $K$ does not admit 
any non-degenerate cycles of length $2$.

We claim that $K$ does not admit any non-degenerate cycles of length $n$ for any $n \ge 3$.
Assume otherwise. Let $n \ge 3$ be the smallest positive integer for which $K$ admits a non-degenerate cycle of length $n$.
Let  $x_1$, $x_2$, $\ldots$, $x_n$ be pairwise distinct numbers in $[0, 1)$ such that
$$ k(x_1, x_2) k(x_2, x_3) k(x_3, x_4) \cdots k(x_{n-1}, x_n) k(x_n, x_1)  \neq 0 . $$
If $1 \le i < j \le n$ with $j-i\le n-2$, it follows from 
$$ k(x_i, x_{i+1}) k(x_{i+1}, x_{i+2}) \cdots k(x_{j-1}, x_j)  \neq 0 $$
and 
$$ k(x_i, x_{i+1}) k(x_{i+1}, x_{i+2}) \cdots k(x_{j-1}, x_j)   k(x_j, x_i) = 0 $$
that $k(x_j, x_i) = 0$.

Let $r > 0$ be small enough so that  $\{[x_i, x_i+r]\}_{i=1}^n$ are pairwise disjoint intervals 
contained in $[0,1]$. Then the standard projections $P_i := P_{x_i, r}$ are pairwise orthogonal.
Set $K_{i, j} := P_i K P_j$. 
Since the operator $(P_1 +\ldots + P_n) K(P_1 +\ldots + P_n)$ is a quasinilpotent trace-class kernel operator on $L^2[0,1]$, we have 
$$
 0 = \tr \left( ((P_1 +\ldots + P_n) K(P_1 +\ldots + P_n))^n  \right) =
\tr \left( \left( \sum_{i=1}^n \sum_{j=1}^n K_{i, j} \right)^{\! n}  \right) = 
$$
\begin{equation}
\label{zero}
 = \sum_{i_1,\ldots,i_n=1}^n \tr (K_{i_1, i_2} K_{i_2, i_3} \cdots K_{i_{n-1}, i_n} K_{i_n, i_1}) .
\end{equation}
The last trace is equal to the integral 
$$ 
\int_{x_{i_1}}^{x_{i_1}+r}  dt_1 \int_{x_{i_2}}^{x_{i_2}+r} dt_2 \cdots  \int_{x_{i_n}}^{x_{i_n}+r} k(t_1, t_2) \, k(t_2, t_3) \, k(t_3, t_4) \cdots k(t_{n-1}, t_n) \, k(t_n, t_1) \,  dt_n . $$
Dividing by $r^n$ and letting $r$ tend to 0 we conclude from (\ref{zero}) that
$$ 0 =   \sum_{i_1,\ldots,i_n=1}^n k(x_{i_1}, x_{i_2}) k(x_{i_2}, x_{i_3}) \cdots 
k(x_{i_{n-1}}, x_{i_n}) k(x_{i_n}, x_{i_1})  . $$
Since $k(x_j, x_i) = 0$ for all $i$, $j$ with $1 \le i \le j \le n$ and $j-i\le n-2$ and 
since $K$ does not admit any non-degenerate cycles of length smaller than $n$, we finally obtain that 
$$ 0 = n \, k(x_1, x_2) k(x_2, x_3) k(x_3, x_4) \cdots k(x_{n-1}, x_n) k(x_n, x_1) .$$
This contradiction shows that $K$ does not admit any non-degenerate cycles.

For every $n \ge 2$ and for every numbers  $x_1$, $x_2$, $\ldots$, $x_n$ in $[0,1]$,  we therefore have
$$ k(x_1, x_2) k(x_2, x_3) \cdots k(x_{n-1}, x_n) k(x_n, x_1)  = 0 . $$
The modulus $|K|$ of $K$ has kernel $|k|$, so that the kernel of $|K|^n$ at point $(x,x) \in [0, 1] \times [0, 1]$
 is equal to the integral 
$$ \int_0^1 dt_1 \int_0^1 dt_2 \cdots  \int_0^1 |k(x, t_1)| \, |k(t_1, t_2)| \, |k(t_2, t_3)| \cdots |k(t_{n-1}, x)| \,  dt_{n-1} 
= 0 . $$
Therefore, $\tr (|K|^n)  = 0$ for all $n \in \bN$. By a well-known theorem (see e. g. \cite[Theorem 14 in Chapter 30]{La}),
we conclude that $|K|$ is a quasinilpotent trace-class positive operator. 
It follows that $|K|$ admits a standard triangularization, by (a corollary to) de Pagter's theorem 
(see, e.g., \cite[Theorem 8.7.9]{RaRo}).  It is clear that the same family of standard subspaces is also invariant under $K$.
This completes the proof. 
\end{proof}

The example following the proof of \cite[Corollary 3.5]{MMR} shows that in Theorem \ref{unit_interval} 
we cannot conclude that $K$ admits a multiplicity-free standard triangularization. 
It is also worth to mention that there exist kernel operators with continuous kernels that are not trace-class operators
(see \cite[Section 30.6]{La}). \\

\section{The disjoint union of the unit interval and atoms}

In this section we extend Theorem \ref{unit_interval} to the measure space $(X, \mu)$ obtained 
by adding $N \in \bN \cup \{\infty\}$ atoms to the unit interval $[0,1]$. To define this more precisely, 
let $A = \{2, 3, 4, \ldots, N+1\}$ if $N \in \bN$, and $A = \bN \setminus \{1\}$ if $N = \infty$. 
We assume that $\mu$ is a Borel measure on $[0,\infty)$ such that
 its support is $X = [0,1] \cup A$, the restriction of $\mu$ to $[0,1]$ is the Lebesgue measure, and 
$\{j\}$  is an atom of measure $1$ for each $j \in A$.
Clearly, the Hilbert space $L^2(X, \mu)$ is the direct sum of $L^2[0,1]$ and $l^2(A)$. 
Let $P_C$ denote the standard projection corresponding to the interval $[0, 1]$, 
and let $P_A$ denote the standard projection corresponding to the set $A$. 

The following proposition was actually proved in \cite[Proposition 3.7]{MMR}. We state it here, because in 
\cite[Proposition 3.7]{MMR} the assumption that $K$ has increasing spectrum is missing.

\begin{proposition}
\label{equal_spectrum}
Let $K$ be a compact operator on $L^2(X, \mu)$ having increasing spectrum relative to standard compressions. 
Then the operators $K$ and $P_A K P_A$ have the same non-zero eigenvalues
with the same algebraic multiplicities, while the operator $P_C K P_C$ is quasinilpotent.
\end{proposition}

Let $K$ be a trace-class kernel operator on $L^2(X, \mu)$ with a continuous kernel $k  : X \times X \to \bC$. Then 
$$ \tr (K) = \tr(P_C K P_C) + \tr(P_A K P_A) =  \int_0^1 k(x,x) dx + \sum_{j=2}^{N+1} k(j,j) = 
\int_X k(x,x) d\mu(x) . $$

Similarly as before,  we say that a kernel operator $K$ on $L^2(X, \mu)$
with a continuous kernel $k$ admits a {\it non-degenerate cycle of length} $n \ge 2$ 
if there exist $n$ pairwise distinct numbers $x_1$, $x_2$, $\ldots$, $x_n$ in $X$ such that
$$ k(x_1, x_2) k(x_2, x_3) k(x_3, x_4) \cdots k(x_{n-1}, x_n) k(x_n, x_1)  \neq 0 . $$

The following theorem extends Theorem \ref{unit_interval}, and it gives a more general answer to Question \ref{open}.

\begin{theorem}
\label{main}
Let $K$ be a trace-class kernel operator on $L^2(X, \mu)$ with a continuous kernel $k$. 
Suppose that $K$ has increasing spectrum relative to standard compressions and that its modulus $|K|$ is also a trace-class 
operator. Then $K$ and $|K|$ admit a (common) standard triangularization.
Furthermore, the operators $K$ and $P_A K P_A$ have the same non-zero eigenvalues
with the same algebraic multiplicities. This holds also for the operators $|K|$ and $P_A |K| P_A$,
while the operators $P_C K P_C$ and $P_C |K| P_C$ are both quasinilpotent.
\end{theorem}

\begin{proof}
We intend to modify the proof of Theorem \ref{unit_interval}.
We still have  $k(x, x) =  0$ for all $x \in [0,1]$, but $\{k(j,j):  j \in A\}$ contains the set of all non-zero eigenvalues of 
$K$ or $P_A K P_A$, by Proposition \ref{equal_spectrum} and Theorem \ref{discrete}.

We claim that $K$ does not admit any non-degenerate cycles of length $2$. In view of Theorem \ref{discrete} 
and Theorem \ref{unit_interval} it is enough to show that $k(x,j) k(j,x) = 0$ for every $x \in [0,1]$ and every $j \in A$.
To this end, let $P_{x, r}$ be the standard projection corresponding to the interval $[x, x+r]$ for some $r \in (0,1)$,
and let $Q_j$ be the standard projection corresponding to the atom $\{j\}$.
If the operators  $(P_{x, r}+Q_j) K (P_{x, r}+Q_j)$ and  $Q_j K Q_j$ are not both quasinilpotent,
they have only one non-zero eigenvalue (that is equal to $k(j,j)$).
Furthermore, $P_{x, r} K P_{x, r}$ is a quasinilpotent trace-class kernel operator, and so we have 
$$ 0 = \tr (((P_{x, r}+Q_j) K (P_{x, r}+Q_j))^2) - \tr ((Q_j K Q_j)^2)  = $$
$$ =  \tr ((P_{x, r}+Q_j) K (P_{x, r}+Q_j) K ) - \tr (Q_j K Q_j K)  = $$
$$ = \tr ((P_{x, r} K P_{x, r})^2) + \tr (P_{x, r} K Q_j K) + \tr (Q_j K P_{x, r} K ) =  $$
$$ = 2 \cdot \tr (P_{x, r} K Q_j K )  = 2 \cdot \int_x^{x+r} k(t,j) k(j,t) dt . $$
Dividing by $r$ and letting $r$ tend to 0 we conclude that $k(x,j) k(j,x) = 0$ as claimed.

Let us prove that $K$ does not admit any non-degenerate cycles of length $n$ for any $n \ge 3$.
Assuming the contrary, let $n \ge 3$ be the smallest positive integer for which $K$ admits a non-degenerate cycle of length $n$.
Let  $x_1$, $x_2$, $\ldots$, $x_n$ be pairwise distinct numbers in $X$ such that
$$ k(x_1, x_2) k(x_2, x_3) k(x_3, x_4) \cdots k(x_{n-1}, x_n) k(x_n, x_1)  \neq 0 . $$
As in  the proof of Theorem \ref{unit_interval}, we have $k(x_j, x_i) = 0$ if $1 \le i < j \le n$ with $j-i\le n-2$, 
since 
$$ k(x_i, x_{i+1}) k(x_{i+1}, x_{i+2}) \cdots k(x_{j-1}, x_j)  \neq 0 $$
and 
$$ k(x_i, x_{i+1}) k(x_{i+1}, x_{i+2}) \cdots k(x_{j-1}, x_j)   k(x_j, x_i) = 0 . $$
Let $m$ be the number of $x_1$, $x_2$, $\ldots$, $x_n$ belonging to $[0,1]$.
In view of  Theorem \ref{discrete}  and Theorem \ref{unit_interval} we can assume that $1 \le m \le n-1$.

Let $r \in (0,1)$ be small enough so that  $\{[x_i, x_i+r]\}_{i=1}^n$ are pairwise disjoint intervals.
Then the standard projections $P_i := P_{x_i, r}$ ($i=1,2, \ldots, n$) are pairwise orthogonal. 
Put $P :=  P_1 +\ldots + P_n$ and  $K_{i, j} := P_i K P_j$.
Now, we have 
$$ \tr \left( (P KP)^n  \right) =
\tr \left( \left( \sum_{i=1}^n \sum_{j=1}^n K_{i, j} \right)^{\! n}  \right) = 
$$
\begin{equation}
\label{trace_sum}
 = \sum_{i_1,\ldots,i_n=1}^n \tr (K_{i_1, i_2} K_{i_2, i_3} \cdots K_{i_{n-1}, i_n} K_{i_n, i_1}) .
\end{equation}
The last trace is equal to the integral 
$$ \int_{x_{i_1}}^{x_{i_1}+r}  d\mu(t_1) \int_{x_{i_2}}^{x_{i_2}+r} d\mu(t_2) \cdots  \int_{x_{i_n}}^{x_{i_n}+r} 
k(t_1, t_2) \, k(t_2, t_3) \, k(t_3, t_4) \cdots k(t_{n-1}, t_n) \, k(t_n, t_1) \,  d\mu(t_n) , $$
which can be non-zero only in two cases:

(a) $i_1 = i_2 = \ldots = i_n$ and $x_{i_1} = j$ for some $j \in A$, in which case it is equal to $k(j,j)^n$;

(b) $(i_1, i_2, \ldots, i_n)$ is a permutation of $(1, 2, \ldots, n)$ that is an $n$-cycle (i.e., a cycle of maximal length). \\
If $C_n$ denotes the set of all $n$-cycles, it follows from (\ref{trace_sum}) that 
\begin{equation}
\label{trace_sum1}
\tr \left( (P KP)^n  \right) = \sum_{j \in \{x_1, \ldots, x_n\} \cap A} k(j,j)^n + 
\sum_{(i_1,\ldots,i_n) \in C_n} \tr (K_{i_1, i_2} K_{i_2, i_3} \cdots K_{i_{n-1}, i_n} K_{i_n, i_1})  .
\end{equation}

By Proposition \ref{equal_spectrum}, the operators $P K P$ and  $P_A P K P P_A$  have the same non-zero eigenvalues
with the same algebraic multiplicities, and these eigenvalues are equal to the set 
$\{k(j,j): j \in \{x_1, \ldots, x_n\} \cap A\} \setminus  \{0\}$ by  Theorem \ref{discrete}.
Therefore, we obtain  from (\ref{trace_sum1}) that 
$$ 0 = \sum_{(i_1,\ldots,i_n) \in C_n} \tr (K_{i_1, i_2} K_{i_2, i_3} \cdots K_{i_{n-1}, i_n} K_{i_n, i_1})  . $$
Dividing by $r^m$ and letting $r$ tend to 0 we conclude that
$$ 0 =   \sum_{(i_1,\ldots,i_n) \in C_n}  k(x_{i_1}, x_{i_2}) k(x_{i_2}, x_{i_3}) \cdots 
k(x_{i_{n-1}}, x_{i_n}) k(x_{i_n}, x_{i_1})  . $$
Since $k(x_j, x_i) = 0$ for all $i$, $j$ with $1 \le i \le j \le n$ and $j-i\le n-2$, we finally obtain that 
$$ 0 = n \, k(x_1, x_2) k(x_2, x_3) k(x_3, x_4) \cdots k(x_{n-1}, x_n) k(x_n, x_1) .$$
This contradiction shows that $K$ does not admit any non-degenerate cycles.

Let $D$ be the (diagonal) kernel operator on $L^2(X, \mu)$ with the kernel obtained from $k$ by multiplication 
with the characteristic function of the set $\{(j,j) : j \in A\}$. 
Thus, the kernel $g$ of the kernel operator $G = K - D$ has all diagonal entries equal zero, and 
for every $n \in \bN$ and for every numbers  $x_1$, $x_2$, $\ldots$, $x_n$ in $X$,  we have
$$ g(x_1, x_2) g(x_2, x_3) \cdots g(x_{n-1}, x_n) g(x_n, x_1)  = 0 . $$
The modulus $|G|$ of $G$ has kernel $|g|$, so that the kernel of $|G|^n$ at point $(x,x) \in X \times X$
 is equal to the integral 
$$ \int_X d\mu(t_1) \int_X d\mu(t_2) \cdots  \int_X |g(x, t_1)| \, |g(t_1, t_2)| \, |g(t_2, t_3)| \cdots 
|g(t_{n-1}, x)| \,  d\mu(t_{n-1}) = 0 . $$
Therefore, $\tr (|G|^n)  = 0$ for all $n \in \bN$. By \cite[Theorem 14 in Chapter 30]{La}),
$|G|$ is quasinilpotent trace-class operator. 
It follows that $|G|$ admits a standard triangularization, by (a corollary to) de Pagter's theorem.
Clearly, the same family of standard subspaces is also invariant under  $|K| = |G| + D$ and $K = G+D$. 
The rest assertions of the theorem  follow from Proposition \ref{equal_spectrum}.
\end{proof}

\section{Finite-rank operators}

Throughout this section, we assume again that $\mu$ is a positive measure on a set $X$  
such that $L^2(X, \mu)$  is a separable complex Hilbert space.
Lemma 3.14 and Proposition 3.15 in \cite{MMR} are the key assertions for the proof of Theorem \ref{finite_rank_MMR}.
However, the conclusion in the last line of the proof of \cite[Lemma 3.14]{MMR} is not established, 
and so we first give a more convincing proof of \cite[Lemma 3.14 and Proposition 3.15]{MMR}.
We then slightly improve Theorem \ref{finite_rank_MMR}.

\begin{proposition}
\label{zero_row}
Let $K$ be a finite-rank operator on $L^2(X, \mu)$ with nilpotent standard compressions.
Then there exist non-zero standard projections $Q$ and $R$ on $L^2(X, \mu)$ such that 
$$ Q K = 0 \ \ \ \textrm{and} \ \ \ K R = 0 . $$
\end{proposition}

\begin{proof}
If the space $L^2(X, \mu)$ is finite-dimensional, this proposition follows easily from Theorem \ref{discrete}.
So, assume that the space $L^2(X, \mu)$ is infinite-dimensional.
It suffices to prove the existence of a non-zero standard projection $Q$ such that $Q K = 0$, 
as then this can be applied for the adjoint operator $K^*$ to obtain a non-zero standard projection $R$ such that $R K^* = 0$,
implying that  $K R = 0$.
 
Let $n$ be the rank of $K$. We may assume that $K \neq 0$, so that $n \ge 1$.
The operator $K$ can be viewed as a kernel operator whose kernel is 
$$ k(x,y) = \sum_{i=1}^n f_i (x) g_i(y) , $$
where each of the sets $\{f_i\}_{i=1}^n$ and $\{g_i\}_{i=1}^n$ is linearly independent in $L^2(X, \mu)$.
Then 
$$ K f = \sum_{i=1}^n \langle f, \overline{g}_i \rangle f_i $$
for all $f \in L^2(X, \mu)$. Note that $K$ is a trace-class operator on $L^2(X, \mu)$, and its trace is equal to
$\tr(K) = \int_X k(x,x) \, d\mu(x)$. To show the latter, we can assume that $\{f_i\}_{i=1}^n$ is an orthonormal basis
for the range of $K$, and so we have 
$$ \tr(K) = \sum_{j=1}^n \langle K f_j, f_j \rangle =  \sum_{j=1}^n \sum_{i=1}^n 
\langle f_j, \overline{g}_i \rangle  \langle f_i, f_j \rangle = 
\sum_{i=1}^n \langle f_i, \overline{g}_i \rangle  = \int_X k(x,x) \, d\mu(x) . $$

Let $P$ be the standard projection corresponding to a measurable subset $E$ of $X$. 
Since the operator $P K P$ is a nilpotent kernel operator on $L^2(X, \mu)$, we have 
$$   0 = \tr (P K P) =  \int_E k(x,x) \, d\mu(x) .  $$
This holds for any measurable subset $E$ of $X$,  so that $k(x,x) = 0$ for almost every $x \in X$. 
It is easy to verify that we may henceforth assume that $k(x,x) = 0$ for all $x \in X$, 
by suitably redefining the functions $\{g_i\}_{i=1}^n$ on sets of measure zero.

Define two maps from $X$ to $\bC^n$ by
$$ F(x) = \left[ \begin{matrix}
f_1(x) & f_2(x)  & \cdots  & f_n(x) 
\end{matrix}  \right]^t  \ \ \ \textrm{and} \ \ \
 G(x) = \left[ \begin{matrix}
g_1(x) & g_2(x)  & \cdots  & g_n(x) 
\end{matrix}  \right]^t  . $$
Then $k(x, y) = F(x)^t G(y)$ for all $x$ and $y$.
Let us denote the support of a function $f : X \to \bC$ by
$$ \supp (f) := \{x \in X : f(x) \neq 0 \} , $$
and define 
$$ \supp (F)  = \bigcup_{i=1}^{n} \supp (f_i) \ \ \ \textrm{and} \ \ 
\supp (G)  = \bigcup_{i=1}^{n} \supp (g_i) \ . $$ 
We must show that there is a non-zero standard projection $Q$ on $L^2(X, \mu)$ such that 
$Q f_i = 0$ for all $i=1, 2, \ldots, n$, or equivalently, the complement 
$(\supp (F))^c$ of $\supp (F)$ has positive measure, as we can then take $Q$ to be 
the standard projection corresponding to this complement.
For $n=1$ this easily follows from the facts that $0 = k(x,x) = f_1 (x) g_1(x)$ for all $x$, and 
$\mu(\supp (f_1)) > 0$ and $\mu(\supp (g_1)) > 0$.
Assume now that $K$ has the smallest rank $n \ge 2$ among all operators 
satisfying the assumptions of the proposition and having the property that $\mu((\supp (F))^c) = 0$.
We can assume that $\supp (G) = X$, since, otherwise, we can replace the operator $K$
with the operator $P_G KP_G$ (considered as an operator on ${\ran P_G}$), 
where  $P_G$ is the standard projection  corresponding to $\supp (G)$.

Define an $n \times n$ matrix of functions on $X$ by 
$$ M(x) = G(x) F(x)^t . $$
Since $F(x)^t G(x) = k(x,x) = 0$, $M(x)$ is a nilpotent rank-one matrix.
For any measurable subset $E$ of $X$ let us define
$$ M(E) = \int_E M(x) \, d\mu(x) . $$
Let $P$ be the standard projection corresponding to a measurable set $E$.
As the operator $P K P$ is a nilpotent finite-rank kernel operator on $L^2(X, \mu)$,
we have
\begin{equation}
\label{trace_zero} 
0 = \tr ((P K P)^2) = \int_E d\mu(x) \int_E k(x,y) k(y,x) d\mu(y) . 
\end{equation}
Since 
$$ \tr(M(x) M(y)) = \tr (G(x) F(x)^t G(y) F(y)^t) =  F(x)^t G(y)  \, \tr (G(x) F(y)^t) = $$
$$  = k(x,y) F(y)^t G(x) =k(x,y) k(y,x) , $$
we obtain that 
\begin{equation}
\label{trace_square} 
\tr ((M(E))^2) = \tr \left( \left( \int_E M(x) \, d\mu(x) \right) \left( \int_E M(y) \, d\mu(y) \right) \right) = 
\end{equation}
$$  \int_E d\mu(x) \int_E \tr (M(x) M(y)) \, d\mu(y) = 
\int_E d\mu(x) \int_E k(x,y) k(y,x) d\mu(y) = 0 ,  $$
by (\ref{trace_zero}).
Now, let $E$ and $F$ be disjoint measurable sets. Since 
$$ \tr ((M(E \cup F))^2) = \tr ((M(E) + M(F))^2) = $$
$$ = \tr ((M(E))^2) + \tr ((M(F))^2) + 2 \, \tr (M(E) M(F)) , $$
(\ref{trace_square}) gives that $\tr (M(E) M(F)) = 0$.

Let $\cV$ be the linear span of the set $\{M(E): E \ {\rm measurable \ set}\}$, and let $m$ be its dimension. Clearly, $m \le n^2$.  Pick measurable sets $E_1$, $E_2$, \ldots, $E_m$ such that 
$M(E_1)$, $M(E_2)$, \ldots, $M(E_m)$ is a basis for $\cV$. As the set function $M$ is finitely additive, we may assume that the sets $E_1$, $E_2$, \ldots, $E_m$ are pairwise disjoint.
Put $E_0 = \cup_{i=1}^m E_i$. Let us show that we can also assume that $\mu(E_0^c) > 0$.
If $\mu(E_0^c) > 0$, we are done, and if $\mu(E_0^c) = 0$, then at least one of the sets $E_1$, $E_2$, \ldots, $E_m$, 
say $E_1$, can be decomposed into two disjoint measurable sets $F_1$ and $F_2$ of positive measure, since the space $L^2(X, \mu)$ is infinite-dimensional.
For $j \in \{1, 2\}$,  let  $\cV_j$ be the linear span of the set 
$\{M(F_j)\} \cup \{M(E_k) : k =2, 3, \ldots, m\}$.   
If the dimensions of $\cV_1$ and $\cV_2$ were both smaller than $m$, then 
$M(F_1)$ and $M(F_2)$ would be linear combinations of matrices $M(E_2)$, $M(E_3)$,\ldots, $M(E_m)$, and 
so $M(E_1) = M(F_1) + M(F_2)$ would also be a linear combination of the same matrices. 
This would contradict the choice of the sets $E_1$, $E_2$, \ldots, $E_m$, and so we must have 
that $\cV_j = \cV$ for some $j \in \{1, 2\}$. Therefore, we can replace $E_1$ by $F_j$, showing that we can assume that  $\mu(E_0^c) > 0$.

Since for every $k \in \{1, \ldots, m\}$ and for every measurable subset $F \subseteq E_0^c$ it holds that
$$ 0 = \tr (M(F) M(E_k)) = \int_F \tr (M(x) M(E_k)) \, d\mu(x)  , $$ 
for almost every $x \in E_0^c$ we have $\tr (M(x) M(E_k)) = 0$ for all  $k \in \{1, \ldots, m\}$.
Therefore, we can find an $x_0 \in E_0^c \cap \supp (F)$ such that 
$\tr (M(x_0) M(E_k)) = 0$ for all  $k \in \{1, \ldots, m\}$.
By linearity of the trace, we have $\tr (M(x_0) \,T) = 0$ for all $T \in \cV$.
Since $M(x_0) = G(x_0) F(x_0)^t$ is a rank-one nilpotent matrix, we can find an invertible $n \times n$ matrix $S$ such that 
$M(x_0) = S^{-1} e_n e_1^t S$, where $e_1$, $e_2$, $\ldots$, $e_n$ are the standard basis vectors of $\bC^n$.
Therefore, $0 = \tr (S^{-1} e_n e_1^t S \,T) =\tr (e_n e_1^t S \,T S^{-1}) = e_1^t (S \,T S^{-1}) e_n$, that is, 
the $(1,n)$-entry of  $S \,T S^{-1}$ is zero for all $T \in \cV$.
Now, define two maps from $X$ to $\bC^n$ by
$$ \tilde{F}(x)^t = \left[ \begin{matrix}
\tilde{f}_1(x) & \tilde{f}_2(x)  & \cdots  & \tilde{f}_n(x)
\end{matrix}  \right]  := F(x)^t S^{-1}  \ \ \ \textrm{and}  $$ 
$$ \tilde{G}(x) = \left[ \begin{matrix}
\tilde{g}_1(x) & \tilde{g}_2(x)  & \cdots  & \tilde{g}_n(x)
\end{matrix}  \right]^t  := S \, G(x) . $$ 
If $T = M(E) \in \cV$ for arbitrary measurable set $E$, then 
$$ S T S^{-1} =  \int_E S G(x) F(x)^t S^{-1}  \, d\mu(x) =  \int_E \tilde{G} (x) \tilde{F}(x)^t  \, d\mu(x) , $$
and so the $(1,n)$-entry of $\tilde{G} (x) \tilde{F}(x)^t$ is zero for almost all $x \in X$,
that is, $\tilde{g}_1 (x) \tilde{f}_n(x) = 0$ for almost all $x \in X$. 
Since $\tilde{F}(x)^t \tilde{G} (x) = F(x)^t S^{-1}  S \, G(x) = k(x, y)$, we may therefore assume already 
for the original functions  that $g_1 (x) f_n(x) = 0$ for almost all $x \in X$. 
Since the functions $f_n$ and $g_1$ are non-zero, there exists a measurable set $E$ such that $\mu(E) > 0$,  
$\mu(E^c) > 0$, $f_n = 0$ on $E$ and $g_1 = 0$ a.e. on $E^c$.
If  $P$ is  the standard projection corresponding to $E$, then the operator $PKP$ has kernel 
$$ k_P(x,y) = \sum_{i=1}^{n-1} f_i (x) g_i(y) , $$
and it is a finite-rank operator on $L^2(E, \mu|_E)$ with nilpotent standard compressions.
By the choice of $n$, there is a set $A \subseteq E$ of positive measure such that 
$f_i = 0$ on $A$ for all $i=1,2, \ldots, n-1$. Since $f_n = 0$ on $A$ as well, this contradiction completes the proof.
\end{proof}

We now slightly improve Theorem \ref{finite_rank_MMR} that is the main result of \cite{MMR}.
We begin with the case of a nilpotent operator.

\begin{theorem}
\label{nilpotent_finite_rank}
Let $K$ be an operator on $L^2(X, \mu)$ of rank $n \in \bN$. 
If $K$ has nilpotent standard compressions, then there exist a positive integer $m \le n+1$ 
and a partition $\{E_1, E_2, \ldots, E_m\}$  of $X$ such that, 
relative to the decomposition $L^2 (X, \mu) = \bigoplus_{j=1}^m L^2 (E_j, \mu|_{E_j})$,
the operator $K$ has the block-matrix form
$$  
K = \left[ \begin{matrix}  
0 & K_{1, 2} & K_{1, 3} &  K_{1, 4} & \ldots & K_{1, m-1} & K_{1, m} \cr
0 & 0 &  K_{2, 3} & K_{2, 4} & \ldots & K_{2, m-1} & K_{2, m}\cr
0 & 0 & 0 & K_{3, 4} & \ldots & K_{3, m-1} & K_{3, m}\cr
0 & 0 & 0 & 0 & \ddots & K_{4, m-1} & K_{4, m}\cr
\vdots & \vdots & \vdots & \ddots & \ddots & \vdots &\vdots \cr
0 & 0 & 0 & 0 & \ldots & 0 & K_{m-1, m} \cr 
0 & 0 & 0 & 0 & \ldots & 0 & 0
\end{matrix}  \right] ,$$
with $K_{j,j+1} \neq 0$ for all $j=1, 2, \ldots, m-1$.  
\end{theorem}

\begin{proof}
We will use induction on $n$. Assume that the theorem holds for operators of ranks strictly less than $n$. 
Define the standard projection on $L^2(X, \mu)$ by
$$ Q_1 = \sup \{ R \in M^\infty(X, \mu) : R^2 = R = R^* , K R = 0 \} . $$
Since $K$ is continuous, we have $K Q_1 = 0$. It follows from Proposition \ref{zero_row} that 
$Q_1 \neq 0$. Let $E_1$ be a measurable set corresponding to $Q_1$, and let $K_1 = (I-Q_1) K (I-Q_1)$. 
Define the standard projection on $L^2(X, \mu)$ by
$$ Q_2 = \sup \{ R \in M^\infty(X, \mu) : R^2 = R = R^* , R \le I-Q_1, K_1 R = 0 \} . $$
Then $K_1 Q_2 = 0$ and  $Q_2 \le I-Q_1$. 
Since $(I-Q_1) K Q_2 = (I-Q_1) K (I-Q_1) Q_2 = K_1 Q_2 = 0$, we have
$Q_1 K Q_2 = K Q_2 \neq 0$, by the definition of $Q_1$.
From $Q_1 K Q_2 \neq 0$ and $K_1 Q_2 = 0$ it follows that \, ${\rm rank}(K_1) \le n-1$.
So, if $n=1$ then $K_1 = 0$, and $\{E_1, E_1^c\}$ is the desired partition of $X$.   
Assume therefore that $K_1 \neq 0$, so that $n \ge 2$. 
By induction hypothesis, there exist a positive integer $m$ 
and a partition $\{E_2, E_3, \ldots, E_m\}$  of $E_1^c$ such that $m-1 \le {\rm rank}(K_1) + 1 \le n$
and, relative to the decomposition $L^2 (E_1^c, \mu|_{E_1^c}) = \bigoplus_{j=2}^m L^2 (E_j, \mu|_{E_j})$,
the restriction of $K_1$ to the range $\ran (I-Q_1) =  L^2 (E_1^c, \mu|_{E_1^c})$ has the block-matrix form
$$  
K_1 = \left[ \begin{matrix}  
0 &  K_{2, 3} & K_{2, 4} & \ldots & K_{2, m-1} & K_{2, m}\cr
0 & 0 & K_{3, 4} & \ldots & K_{3, m-1} & K_{3, m}\cr
0 & 0 & 0 & \ddots & K_{4, m-1} & K_{4, m}\cr
\vdots & \vdots & \ddots & \ddots & \vdots &\vdots \cr
0 & 0 & 0 & \ldots & 0 & K_{m-1, m} \cr 
0 & 0 & 0 & \ldots & 0 & 0
\end{matrix}  \right] ,$$
with $K_{j,j+1} \neq 0$ for all $j=2, 3,\ldots, m-1$.  
Since $K_{1, 2} := Q_1 K Q_2 \neq 0$, the partition  $\{E_1, E_2, \ldots, E_m\}$ of $X$ has the desired properties.
\end{proof}

\begin{theorem}
\label{finite_rank}
Let $K$ be an operator on $L^2(X, \mu)$ of rank $n  \in \bN$.
If $K$  has increasing spectrum relative to standard compressions, then it admits a standard, multiplicity-free triangularization. 
Furthermore, there exist a positive integer $m \le 2 n+1$ 
and a partition $\{E_1, E_2, \ldots, E_m\}$  of $X$ such that,
relative to the decomposition $L^2 (X, \mu) = \bigoplus_{j=1}^m L^2 (E_j, \mu|_{E_j})$,
the operator $K$ has the block-matrix form
$$  
K = \left[ \begin{matrix}  
K_{1, 1} & K_{1, 2} & K_{1, 3} &  K_{1, 4} & \ldots & K_{1, m-1} & K_{1, m} \cr
0 & K_{2, 2} &  K_{2, 3} & K_{2, 4} & \ldots & K_{2, m-1} & K_{2, m}\cr
0 & 0 & K_{3, 3} & K_{3, 4} & \ldots & K_{3, m-1} & K_{3, m}\cr
0 & 0 & 0 & 0 & \ddots & K_{4, m-1} & K_{4, m}\cr
\vdots & \vdots & \vdots & \ddots & \ddots & \vdots &\vdots \cr
0 & 0 & 0 & 0 & \ldots & K_{m-1, m-1} & K_{m-1, m} \cr 
0 & 0 & 0 & 0 & \ldots & 0 & K_{m, m}
\end{matrix}  \right] ,$$
where each diagonal block $K_{j,j}$ can be non-zero only when $L^2 (E_j, \mu|_{E_j})$ is a one-dimensional space 
(corresponding to an atom), and in this case $K_{j,j}$ is a non-zero eigenvalue of $K$.
\end{theorem}

\begin{proof}
In view of Theorem \ref{nilpotent_finite_rank} we may assume that $K$ has at least one non-zero eigenvalue.
Let $\lambda_1$, $\ldots$, $\lambda_r$ be non-zero eigenvalues of $K$, with algebraic multiplicities 
$m_1$, $\ldots$, $m_r$, respectively.  For each $i=1, \ldots, r$, let $M_{i,1}$,  $\ldots$, $M_{i,m_i}$ be 
the standard projections (corresponding to the atoms $A_{i,1}$,  $\ldots$, $A_{i,m_i}$, respectively) 
for which $M_{i,j} K M_{i,j} = \lambda_i M_{i,j}$ for all $j=1, \ldots, m_i$. The existence of these atoms is 
guaranteed by \cite[Corollary 3.8]{MMR}.
Now we apply  \cite[Corollary 3.12]{MMR} to conclude that the finite-rank operator 
$$ G = K - \sum_{i=1}^r \sum_{j=1}^{m_i} \lambda_i M_{i,j} $$ 
has nilpotent standard compressions. 
Since ${\rm rank}(K) = n$ and $\sum_{i=1}^r m_i \le n$, it holds that \, ${\rm rank}(G) \le 2 n$.

By Theorem \ref{nilpotent_finite_rank}, there exist a positive integer $m \le 2 n+1$ 
and a partition $\{E_1, E_2, \ldots, E_m\}$  of $X$ such that,
relative to the decomposition $L^2 (X, \mu) = \bigoplus_{k=1}^m L^2 (E_k, \mu|_{E_k})$,
the operator $G$ has the block-matrix form
$$  
G = \left[ \begin{matrix}  
0 & G_{1, 2} & G_{1, 3} &  G_{1, 4} & \ldots & G_{1, m-1} & G_{1, m} \cr
0 & 0 &  G_{2, 3} & G_{2, 4} & \ldots & G_{2, m-1} & G_{2, m}\cr
0 & 0 & 0 & G_{3, 4} & \ldots & G_{3, m-1} & G_{3, m}\cr
0 & 0 & 0 & 0 & \ddots & G_{4, m-1} & G_{4, m}\cr
\vdots & \vdots & \vdots & \ddots & \ddots & \vdots &\vdots \cr
0 & 0 & 0 & 0 & \ldots & 0 & G_{m-1, m} \cr 
0 & 0 & 0 & 0 & \ldots & 0 & 0
\end{matrix}  \right] ,$$
with $G_{j,j+1} \neq 0$ for all $j=1, 2, \ldots, m-1$.  
The block-matrix form of $K$ differs from the last one only on some diagonal blocks, which are restrictions of some 
$\lambda_i M_{i,j}$ to $L^2 (E_k, \mu|_{E_k})$. Clearly, this block-matrix form of $K$ is not necessarily the desired one.

We will complete the proof of the theorem with induction on $n$. 
Assume that the theorem holds for operators of ranks strictly less than $n$.  
Among the sets  $E_1$, $\ldots$, $E_m$ there is exactly one, say $E_p$, that contains the atom $A_{1,1}$.
Let $F_1 = \cup_{j=1}^{p-1} E_j$ and $F_2 = (E_p \setminus A_{1,1}) \cup (\cup_{j=p+1}^{m} E_j)$.
It follows from the above block-matrix form of $K$ that, relative to the decomposition 
$$ L^2 (X, \mu) = L^2 (F_1, \mu|_{F_1}) \oplus L^2 (A_{1,1}, \mu|_{A_{1,1}}) \oplus L^2 (F_2, \mu|_{F_2}) , $$
the operator $K$ has the block-matrix form
$$  
K = \left[ \begin{matrix}  
K_1 & * & * \cr
0 & \lambda_1 &  * \cr
0 & 0 & K_2 
\end{matrix}  \right] . $$
Now, if $n=1$ then $K_1 = 0$ and $K_2=0$, so that the theorem holds in this case.
Assume therefore that $n \ge 2$. Then, for each $j \in \{1, 2\}$, the operator $K_j$ on $L^2 (F_j, \mu|_{F_j})$ 
has the desired block-matrix form with at most  $(2 \, {\rm rank}(K_j) + 1)$ diagonal blocks, by the induction hypothesis.
Therefore, the operator $K$ has  the block-matrix form in which the number of diagonal blocks is at most 
$$ (2 \, {\rm rank}(K_1) + 1) + 1 + (2 \, {\rm rank}(K_2) + 1) = 
      2 \, ({\rm rank}(K_1) + 1 + {\rm rank}(K_2)) + 1 \le 2 n + 1 . $$
This completes the proof of the theorem.
\end{proof}

We complete this paper by an example showing that the bound $2 n +1$ in Theorem \ref{finite_rank} cannot be improved.

\begin{example} {\rm 
Let $n \in \bN$,  and let $e_1$, $e_2$, $\ldots$, $e_{2 n+1}$ be the standard basis vectors of $\bC^{2 n+1}$.
For each $j=1,2, \ldots, n$, let $f_j = \sum_{i=2 j}^{2n+1} e_i$. Define 
$$ K = \sum_{j=1}^n (e_{2j-1} + e_{2 j}) \cdot f^t_j . $$
For example, if $n=2$ then
$$ K = \left[ \begin{matrix}  
0 & 1 & 1 &  1 & 1 \cr
0 & 1 & 1 &  1 & 1 \cr
0 & 0 & 0 &  1 & 1 \cr
0 & 0 & 0 &  1 & 1 \cr
0 & 0 & 0 &  0 & 0 
\end{matrix}  \right] . $$
Then $K$ is an upper triangular matrix of rank $n$, and so it has increasing spectrum relative to standard compressions.
Furthermore, it already has the form guaranteed by Theorem \ref{finite_rank}, and it is easy to see that 
there is no such form with a number of diagonal blocks smaller than $2n+1$. }
\end{example}

\vspace{2mm}
{\bf
\begin{center}
 Acknowledgment.
\end{center}
} 
The author was supported by grant P1-0222 of the Slovenian Research Agency.

\vspace{2mm}

\noindent
Roman Drnov\v sek \\
Department of Mathematics \\
Faculty of Mathematics and Physics \\
University of Ljubljana \\
Jadranska 19 \\
SI-1000 Ljubljana \\
Slovenia \\
e-mail : roman.drnovsek@fmf.uni-lj.si 

\end{document}